\documentclass[12pt, letterpaper]{amsart}
\usepackage[margin=1in]{geometry}
\usepackage{pb-diagram,amsmath,amssymb,epsfig,tikz,color, enumitem} 
\usepackage{amsthm}
\usepackage{amssymb}
\usepackage{mathtools}
\usepackage{tikz}
\usepackage{tikz-cd}
\usepackage{graphics}
\usepackage{lscape,graphicx}
\usepackage{amsfonts}
\usepackage{longtable}
\usepackage{epsfig}
\usepackage{float}
\usepackage{rotating}
\usepackage{adjustbox}
\usepackage{fancyhdr}
\usepackage[titletoc]{appendix}
\usepackage{multicol}
\usepackage{anyfontsize}
\usepackage{lipsum}
\usepackage{setspace}
\usepackage{xcolor}
\usepackage[english]{babel}
\usepackage{ragged2e}
\usepackage{blindtext}
\usepackage{bigints}
\usepackage{multirow}
\usepackage{fontenc}
\usepackage{dirtytalk}
\usepackage{tabularray}
\usepackage{epsfig}
\usepackage{verbatim}
\usepackage{hyperref}
\usepackage{mathrsfs}
\usepackage{csquotes}
\usetikzlibrary{petri}
\usepackage{wrapfig}
\usetikzlibrary{positioning}
\usepackage{microtype}  
\usepackage{array}
\usepackage{pdflscape}

\AtBeginDocument{\justifying \setlength{\emergencystretch}{1.2em}}

\definecolor{purple}{rgb}{0.59, 0.44, 0.84}

\def\PSL{\text{PSL}}
\def\SL{\text{SL}}

\title{Exceptional Sets for Certain ${}_2F_1$ Hypergeometric Functions}
\author{Archisman Bhattacharjee}
\address{Department of Mathematics, Louisiana State University, Baton Rouge, LA 70803, USA}
\email{abhat25@lsu.edu}


\usepackage{macros}
\begin{document}

\begin{abstract}
For $a,b,c \in \Q$, the exceptional set associated to the Gauss hypergeometric function $_2F_1(a,b,c;z)$ is defined by
$$
E(a,b,c) := \{ z \in \overline{\Q} \mid {}_2F_1(a,b,c;z) \in \overline{\Q} \}.
$$
In this paper, the exceptional sets $E(a,b,c)$ are determined explicitly for each 
$_2F_1(a,b,c;z)$ whose monodromy group is an arithmetic triangle group 
in Takeuchi’s class I. The description is obtained via hypergeometric--modular identities together with transcendence results for periods of abelian varieties due to W\"ustholz, and classical result of Schneider on algebraic values of $j$-invariant of elliptic curves with complex multiplication.
\end{abstract}
\maketitle
\tableofcontents
\section{Introduction} 

Given $a,b,c \in \Q$ and $-c \notin \N \cup \{0\}$, the Gauss hypergeometric function is defined by the convergent power series
$$
_2F_1(a,b,c;z) = \;\; \hyp(a,b,c,z) \coloneqq \sum_{k=0}^{\infty}\frac{(a)_k (b)_k}{(c)_k} \frac{z^k}{k!},
\;\;\;\; |z|<1,
$$
where $(a)_0=1$, $(a)_k = a(a+1)\dots(a+k-1)$.  This is a holomorphic local solution around $z=0$ of the hypergeometric differential equation 
\begin{equation} \label{eq00}
\text{HDE}(a,b,c;z) :\;\; y''(z) + \frac{c-(a+b+1)z}{z(1-z)}y'(z) -\frac{ab}{z(1-z)}y(z) = 0,
\end{equation}
which has three regular singularities $0,1,\infty$. When $|z| >1$, we use the same notation for an analytic continuation of the normalized ($_2F_1(a,b,c;0)=1$) local solution along a chosen path in $\mathbb{CP}^1\setminus\{0,1,\infty\}$. 
Near a non-singular point $z_0$, the solution space $S(a,b,c;z_0)$ of this differential equation \eqref{eq00} gives a homomorphism 
$$
\pi_1(\mathbb{CP}^1 \setminus \{0,1,\infty\}, z_0) \rightarrow \text{GL}(S(a,b,c;z_0)),
$$ 
called the monodromy representation, which is determined up to conjugation. The image of the induced projective representation is called the projective monodromy group, which is a triangle group, defined in Section~\ref{sec:basic}. In what follows, we refer to this projective monodromy group simply as the monodromy group. It is called \emph{arithmetic} if it is an arithmetic subgroup of $\mathrm{PSL}_2(\R)$ as described in Takeuchi~\cite{takeuchi1977arithmetic}; equivalently, it is commensurable with the norm-one group of an order in a quaternion algebra over a totally real number field.
\par In~\cite{beukers1989monodromy}, Beukers and Heckman proved that the finiteness of the monodromy implies that the hypergeometric function is algebraic. These cases were classically classified by Schwarz through his list of finite monodromy hypergeometric equations; see~\cite{schwarz1873hypergeometrische}. 
But in general, the $_2F_1$ is not always an algebraic function of $z$. 
For non-algebraic cases, in~\cite{wolfart1988werte}, Wolfart studied the exceptional set for the datum $(a,b,c)$, which is given by
$$
E(a,b,c) := \{z \in \overline{\Q} \;\;\mid\;\; _2F_1(a,b,c;z) \in \overline{\Q}\}.
$$
He conjectured that the exceptional set is infinite if and only if the monodromy group is arithmetic. He proved this conjecture for the case where the projective monodromy group is an arithmetic triangle group by showing that the exceptional set is Zariski dense in $\C$, and hence infinite. Takeuchi classified the $85$ arithmetic triangle groups in~\cite{takeuchi1977commensurability}.
Later, Cohen and W\"ustholz~\cite{cohen2002applications} proposed a weaker version of the Andr\'e--Oort conjecture
\cite[Conjecture~1.1]{edixhoven2003subvarieties} to approach this problem.  The Zariski density of exceptional sets for higher-dimensional analogues of Appell--Lauricella hypergeometric series was proved by Desrousseaux--Tretkoff--Tretkoff in~\cite{desrousseaux2008zariski}. See~\cite[Chapter 5]{tretkoff}, for the works of Edixhoven--Yafaev~\cite{edixhoven2003subvarieties}, Desrousseaux--Tretkoff--Tretkoff~\cite{desrousseaux2008zariski, desrousseaux2004valeurs, tretkoff2008transcendence, tretkoff_2011_transcendence} in this direction.  
Archinard explicitly determined exceptional sets for hypergeometric series with monodromy group isomorphic to $\mathrm{PSL}_2(\Z)$ in~\cite{archinard2003exceptional}, using her earlier construction of the associated abelian varieties in~\cite{archinard2000abelian}. Baba--Granath studied exceptional sets arising from hypergeometric functions associated to Shimura curves and quaternionic modular forms in~\cite{baba2015quaternionic}; see also their work on Picard--Fuchs differential equations attached to Shimura curves over $\Q$ in~\cite{baba2021picard}. Yang studied special values of hypergeometric functions through periods of CM elliptic curves in~\cite{yang2018special}. Since the Class I of Takeuchi's classification contains $\mathrm{PSL}_2(\Z) \simeq (2,3,\infty)$ and other groups commensurable with it, a natural question arises:

\medskip
\noindent
\textit{If the monodromy group corresponding to $(a,b,c)$ is isomorphic to one of the nine non-compact arithmetic triangle groups commensurable with $\mathrm{PSL}_2(\Z)$, what are the exceptional sets $E(a,b,c)$?}

\medskip

$$
  \begin{diagram}
  \node[2]{(2,6,\infty)}  \arrow{sw,l,-}{2}  \arrow{s,l,-}{2}
  \node[1]{(2,3,\infty)}   \arrow{sw,l,-}{4}  \arrow{s,l,-}{2} \arrow{se,l,-}{3} 
   \node[1]{(2,4,\infty)}   \arrow{s,l,-}{2}  \arrow{se,l,-}{2} \\
   \node[1]  {(6,6,\infty)} 
   \node[1]  {(3,\infty,\infty)} 
    \node[1]  {(3,3,\infty)} \arrow{se,l,-}{3}
   \node[1]   {(2,\infty,\infty)} \arrow{s,r,-}2 
   \node[1]  {(4,4,\infty)} \\ 
   \node[4]{(\infty,\infty,\infty)}
  \end{diagram}
$$
\begin{center}Figure 1 \label{fig1}: Takeuchi's Class I \end{center}

\par Our main result can be summarized as follows. Denote by $\UH = \{x+iy \;\mid  x,y\in \R, \; y>0\}$ the complex upper half-plane.

\begin{theorem} \label{tmain}
Let $(a,b,c)$ be a datum satisfying $0<a,b,c<1$ such that the monodromy group $\Delta$ associated to $\text{HDE}(a,b,c;z)$ is an arithmetic triangle group in Takeuchi's class I (see Figure \ref{fig1}). 
Then there exist a modular function $t$ for the corresponding group $\Delta$, and a positive integer $d \in \{1,2,3\}$ such that
$$
z \in E(a,b,c)\; \text{if and only if}\; z = t(\tau_d) \;\text{for some}\; \tau_d \in \Q(\sqrt{-d}) \cap \UH.
$$
For the remaining class I data with $c \geq 1$, $E(a,b,c) = \{0\}$.
\end{theorem}

\vspace{0.3cm}
In this article, we explicitly determine the exceptional sets for all such $(a,b,c)$. They are of the form
$$
E(a,b,c) = \{z \in \overline{\Q} \mid \exists\; \tau_d \in \Q(\sqrt{-d}) \cap \UH\;\; \text{with}\;\; z = t (\tau_d)\},
$$
where $t$ is a Hauptmodul chosen for each triangle group in Table~\ref{table:2}. 
The integers $d$ and modular functions $t$ are explicitly determined for all such $(a,b,c)$ in Table~\ref{table:main}. For example, the datum \(\left(\frac14,\frac14,\frac12\right)\) corresponds to the triangle group \((2,\infty,\infty)\), identified with the \(\Gamma_0(2)\)-modular curve, and the relevant CM field is \(\Q(i)\). We have 
$$
E(\frac{1}{4},\frac{1}{4},\frac{1}{2}) = \{z \in \overline{\Q} \mid \exists\; \tau_d \in \Q(i) \cap \UH\;\; \text{with}\;\; z = S_2 (\tau_d)\},
$$ 
where $S_2$ is a modular function for $\Gamma_0(2)$, defined in Table~\ref{table:2}. Taking \(\tau=i\), we  can also give explicit algebraic values like $S_2(i)=9$ and
$$
{}_2F_1\!\left(\frac14,\frac14,\frac12;9\right)=\frac{2-i}{2\sqrt2}.
$$
More such explicit examples are listed in Table \ref{table:explicit_examples}. The same CM point $\tau=i$ as in the above example also appears in \cite{AroraMondalNakagawaTu2026}. The authors define
$u(\tau):=-64\,\eta(2\tau)^{24}/\eta(\tau)^{24}$ which is related to $S_2$ by $u=1/(1-S_2)$,
one of the cross-ratio transformations of Section~\ref{sec:hypergeom}. So the value
$S_2(i)=9$ gives $u(i)=-1/8$, which is the CM-value of the Hauptmodul used in
\cite[Table~1(A)]{AroraMondalNakagawaTu2026} (with $\tau_d=i$). This is used to evaluate the special
$L$-value of the weight-3 CM Hecke eigenform \texttt{f64.3.c.a}. The explicit CM-point
evaluations of Hauptmoduln obtained in this paper (Table~\ref{table:explicit_examples})
can therefore supply algebraic inputs for such special $L$-value computations. The authors also find the explicit transcendental values of 
${}_3F_2\!\left(\frac12,\frac12,\frac12;1,1;u(\tau_d)\right)$ at different CM points $\tau_d$ which equals ${}_2F_1\!\left(\frac14,\frac14,1;u(\tau)\right)^2$ due to Clausen's Formula~\cite[Section 2.5]{slater1966hypergeometric}. 
\par In this paper, the main idea is to relate values of the Gauss hypergeometric function to periods using Euler’s integral representation. This leads to algebraic curves and their associated abelian varieties, providing a geometric setting in which the $_2F_1$ is interpreted in terms of periods. In this framework, the results of Wolfart and W\"ustholz~\cite{wolfart1988werte, wolfart1985uberlagerungsradius} give necessary conditions for the algebraicity of such quantities. 
\par The hypergeometric functions considered in this paper also admit modular form oriented identities. By expressing $_2F_1(a,b,c;t(\tau))$ in terms of modular forms, we obtain explicit formulas that can be evaluated at CM points. These modular descriptions complement Wolfart and W\"ustholz's transcendence results and allow us to identify the values of $z$ for which $_2F_1(a,b,c;z)$ is algebraic. Combining these two approaches, we determine the exceptional sets for the data arising from Takeuchi’s class I.
\par A detailed study of the relationship between hypergeometric functions and modular forms has been developed in recent years. Earlier connections appeared through Gaussian hypergeometric series and supercongruences, for example, in the work of Ahlgren--Ono and Kilbourn~\cite{ahlgren2000gaussian,kilbourn2006extension}. More recently, Allen--Grove--Long--Tu developed the explicit hypergeometric--modularity method, which associates modular forms and modular Galois representations to suitable hypergeometric data~\cite{allen2025explicit,allen2025explicit2}. This framework has also inspired further explicit constructions of hypergeometric modularity, such as~\cite{maity2026explicithypergeometricmodularitycertain}.
\par On the geometric side, abelian varieties arising from generalized Legendre curves have been studied in~\cite{deines2016generalized}, where their periods, Galois representations, and endomorphism algebras are considered. These ideas have also found applications in the study of special values of modular forms and $L$-functions; see for example~\cite{AroraMondalNakagawaTu2026, Rosen2024Lvalues,  rosen2025modular}. In particular, the latter work relates special $L$-values of certain CM Hecke eigenforms to hypergeometric series via Ramanujan’s theory of elliptic functions to alternative bases and the modularity of hypergeometric Galois representations. 
\par Hypergeometric functions over finite fields are also an important object in arithmetic geometry. See~\cite{fuselier2022hypergeometric} for a systematic study of hypergeometric functions over finite fields. Such functions have also been used to study traces of Hecke operators. Formulas of this type were obtained in~\cite{ahlgren2000gaussian, ahlgren2003gaussian,ono2007values}. More recently, Hoffman--Li--Long--Tu gave a general geometric approach for arithmetic triangle groups, expressing traces of Hecke operators on spaces of cusp forms in terms of hypergeometric character sums over finite fields~\cite{hoffman2024traces}.
\par The present work is motivated by these developments and focuses on determining algebraicity and exceptional sets for certain $_2F_1$ hypergeometric functions. The paper is organized as follows. In Section~\ref{sec:basic},  we introduce the algebraic curves associated to Euler’s integral and describe the corresponding abelian varieties. In Section~\ref{sec:hypergeom}, we derive identities relating the hypergeometric functions to modular forms via suitable Hauptmoduln. In Sections~\ref {sec:ex1} and~\ref{sec:ex2}, we apply transcendence results to study algebraicity and determine the exceptional sets. The appendix contains tables summarizing the results and explicit examples of algebraic values. 
\medskip

\textbf{Acknowledgments.} The author is grateful to Ling Long and Fang-Ting Tu for helpful discussions on the subject. He would also like to thank Gene Kopp, Esme Rosen, Paresh Arora, and Abhishek Dangodara for valuable feedback on improving the presentation of this paper. The author is supported by a summer research assistantship from the Department of Mathematics, Louisiana State University.
\medskip
\section{Abelian Varieties Associated to Hypergeometric Data} \label{sec:basic}
We begin with Euler’s integral representation of the hypergeometric function and view it in a geometric setting. The integrands naturally define differential forms on algebraic curves, whose desingularizations give rise to abelian varieties. Following Archinard~\cite{archinard2000abelian}, this allows us to study the associated periods in a suitable framework.
\par Consider the Euler integral representation of the hypergeometric function, which allows us to interpret the integral in a geometric setting.
For $c>b>0$, Euler showed that
\begin{equation}\label{euler}
\hyp(a,b,c,z) = \dfrac{1}{B(b,c-b)} \int_0^1 x^{b-1} (1-x)^{c-b-1} (1-zx)^{-a} dx  
\end{equation}
where $B(r,s) = \int_0^1 x^{r-1} (1-x)^{s-1}{ dx}$. Let L.C.D denote the least positive common denominator for a set of fractions, and
$$N= \text{L.C.D}(a,b,c),\; A= N(1-b),\; B=N(1+b-c),\; C=Na.$$ 
We assume that $N \nmid A, B, C, A+B+C$. The integrand in the numerator $x^{b-1} (1-x)^{c-b-1} (1-zx)^{-a}dx$ is a differential 1-form on the projective variety defined by 
\begin{equation} \label{var1}
C(N,z): y^N = x_1^{N-A-B-C}x_2^A(x_1-x_2)^B(x_1-zx_2)^C.  
\end{equation}
Similarly, let 
$$M=\text{L.C.D}(b,c), \;P=M(1-b),\; Q=M(1+b-c).$$
The integrand $x^{b-1} (1-x)^{c-b-1}{ dx}$ appearing in the definition of $B(b,c-b)$ is a differential 1-form on the projective variety defined by 
\begin{equation} \label{var2}
C(M,0): y^M = x_1^{M-P-Q}x_2^P(x_1-x_2)^Q. 
\end{equation}

These projective varieties can be viewed as curves in $\mathbb{P}^2(\C)$. 
In~\cite{archinard2000abelian}, Archinard constructed $X(N,z)$, the desingularization of the projective curve defined by $C(N,z)$, by gluing local desingularizations. 
\par For a given $n$, with $1 \leq n \leq N$, let $\omega_n = \frac{x^{\alpha}(1-x)^{\beta}(1-zx)^{\gamma}dx}{y^n}$ be a differential form on $C(N,z)$ for some $\alpha, \beta, \gamma \in \Z$. For the desingularization map $\pi_z: X(N,z) \rightarrow C(N,z)$, the pullback $\pi_z^*\omega_n$ is holomorphic on $X(N,z)$ if and only if the following inequalities are satisfied. 
\begin{equation}\label{eq1}
\begin{aligned}
\alpha & \geq \frac{(N,A)+nA}{N} -1 ,\\  
\beta & \geq \frac{(N,B)+nB}{N} -1 ,\\ 
\gamma & \geq \frac{(N,C)+nC}{N} -1 ,\\
\alpha + \beta + \gamma & \leq \frac{n(A+B+C)-(N, N-A-B-C)}{N} -1.\\
\end{aligned} 
\end{equation}
Here $(N, A)$ denotes the g.c.d of $N$ and $A$. These inequalities are obtained by investigating whether the local parameterizations of $X(N,z)$ are algebraic morphisms.
\par To apply the transcendental results by Wolfart and W\"ustholz~\cite{wolfart1985uberlagerungsradius}, we need abelian varieties where $\int \pi_z^*\omega$ lives as a period for the differential forms $\omega$ appearing as integrand in \eqref{euler}. The Jacobian variety of $X(N,z)$ is defined by 
\begin{equation}\label{eqJac1}
\text{Jac}(X(N,z)) = \Omega^1[X(N,z)]^{\text{*}}/ \iota (H_1(X(N,z)(\C), \Z))     
\end{equation} where $(\iota[\gamma]) (\omega) = \int_\gamma \omega$. 
We will restrict ourselves to consider an abelian subvariety $\text{Jac}_{\text{new}}(X(N,z))$ of  $\text{Jac}(X(N,z))$ defined in~\cite{archinard2000abelian}. Since the group $\mu_N$ of $N$th roots of unity acts on $X(N,z)$, by
$$
\zeta_N \cdot (x,y) = (x,\zeta_N^{-1}\, y), \;\;\;\;\;\; \zeta_N \in \mu_N,
$$
it also acts on the space of holomorphic differentials $\Omega^1(X(N,z))$. Hence $\Omega^1(X(N,z))$ decomposes into isotypical components
$$
\Omega^1(X(N,z))=\bigoplus_{i=1}^{N-1} V_i,
$$
where $V_i$ is the eigenspace on which $\zeta_N \in \mu_N$ acts by multiplication by $\zeta_N^i$. Following~\cite{archinard2000abelian}, we define $\text{Jac}_{\mathrm{new}}(X(N,z))$ to be the abelian subvariety of $\text{Jac}(X(N,z))$ whose space of holomorphic differentials is
$$
\Omega^1(\text{Jac}_{\mathrm{new}}(X(N,z)))=\bigoplus_{\substack{1\le i\le N-1\\ (i,N)=1}} V_i.
$$
In particular,
$$
\dim \text{Jac}_{\mathrm{new}}(X(N,z))
=
\sum_{\substack{1\le i\le N-1\\ (i,N)=1}} \dim V_i.
$$
Moreover, for $(i,N)=1$, we have
$$
\dim V_i+\dim V_{N-i}=2,
$$
and therefore $\dim \text{Jac}_{\mathrm{new}}(X(N,z))=\phi(N)$,  where $\phi$ is the Euler's totient function.
For $c \notin \Z$, corresponding to $B(b,c-b)$, we similarly consider for $1 \leq n \leq M$, the differential form $\eta_n = \frac{x^{\alpha}(1-x)^{\beta}dx}{y^n}$ for some $\alpha, \beta \in \Z$. The holomorphic condition of the pullback of $\eta_n$ on $X(M,0)$ are given by inequalities
\begin{equation}\label{eq2}
\begin{aligned}
\alpha & \geq \frac{(M,P)+nP}{M} -1, \\  
\beta & \geq \frac{(M,Q)+nQ}{M} -1 ,\\ 
\alpha + \beta & \leq \frac{n(P+Q)-(M, M-P-Q)}{M} -1.\\
\end{aligned} 
\end{equation}
The dimension of the new Jacobian variety of $X(M,0)$ is $\frac{\phi(M)}{2}$. 
\par For a compact Riemann surface $X$ of genus $g >0$, $\text{Jac}(X)$ can be identified with $\C^g/\Lambda$ where $\Lambda = \Z A_1 \oplus \Z A_2 \oplus \dots \oplus\Z A_{2g}$ is the period lattice. Here, $A_i = (\int_{a_i} \omega_1, \int_{a_i} \omega_2, \dots, \int_{a_i} \omega_g )$ where  $\{a_i\}_{1\leq i \leq 2g}$ is a $\Z$-basis of $H_1(X,\Z)$ and $\{\omega_i\}_{1 \leq i \leq g }$ is a basis of $\Omega^1_{\text{an}}[X]$. Due to a similar construction over $\Z[\zeta_N]$, there is an isogeny of algebraic varieties
$$\text{Jac}_{\text{new}}(X(N,z)) \sim \C^{\phi(N)}/ \Lambda(z)$$
for 
{\small
\[
\Lambda(z)=
\left\{\Bigl(
  \sigma_i(u)\,\int_{\gamma_{0 1}}\!\omega_i
  \;+\;
  \sigma_i(v)\,\int_{\gamma_{\frac{1}{z}\infty}}\!\omega_i
\Bigr) : (i,N) = 1, \omega_i \;\text{is a basis element of}\;V_i \; ; \; u,v\in \Z[\zeta_N] \right\},
\]
}
where $\sigma_i : \Z[\zeta_N] \rightarrow \Z[\zeta_N]$ , defined by $\sigma_i(\zeta_N) = \zeta_N^i$, is an automorphism, and the set of Pochammer loops $\{\gamma_{0 1}, \gamma_{\frac{1}{z}\infty}\}$ (see~\cite[Section 2.1]{fuselier2022hypergeometric}) is a generator of a rank-$2$ $\Z[\zeta_N]$-submodule of $H_1(\text{Jac}_{\text{new}}(X(N,z)),\Z)$.

\section{Hypergeometric Modular identities} \label{sec:hypergeom}
\subsection{Schwarz triangle group} We begin by recalling the hyperbolic structure on the upper half-plane and the associated Schwarz triangle function arising from the hypergeometric differential equation. This gives a geometric realization of triangle groups. We also describe explicit Hauptmoduln for the non-compact arithmetic triangle groups in Takeuchi’s class I, which will be used later to relate hypergeometric functions to modular functions.
\par The upper half plane, endowed with a hyperbolic metric
$$(ds)^2 = \frac{(dx)^2 + (dy)^2}{y^2},\;\;\;\;\tau = x + i y \in \UH$$ is called the hyperbolic plane.
For a given $(a,b,c)$, define $$p = \frac{1}{e_1} = |1-c|,\;\;q = \frac{1}{e_2} =|c-a-b|,\;\;r= \frac{1}{e_3} = |a-b|,$$
where $e_1,e_2,e_3\in \N \cup \{\infty\}$. Here we adopt the convention that $\frac{1}{\infty}=0$. For this $(a,b,c)$ let $f,g$ be two linearly independent solutions to the $\text{HDE}(a,b,c;z)$ \eqref{eq00} at a point $z \in \UH$. We define  the Schwarz map 
\begin{equation}\label{Schwarz} D_{abc}(z) = \frac{f(z)}{g(z)}. \end{equation} 
If $p+q+r < 1$, we can pick $f$ and $g$ in a way such that $D_{abc}$ gives a bijection from $\UH \cup \R$ onto a hyperbolic triangle (curvilinear triangle in hyperbolic plane), with vertices $D(0)$, $D(1)$, $D(\infty)$ and corresponding angles $p\pi, q\pi, r\pi$. By Schwarz reflection (see~\cite[Proposition~7.1 and Corollary~7.2]{yoshida1997hypergeometric}) across each side of the image triangle, corresponding to
one of the intervals \((-\infty,0)\), \((0,1)\), and \((1,\infty)\), the
Schwarz map extends to the lower
half-plane and maps bijectively onto the adjacent reflected triangle. Let $\tau_p,\tau_q,\tau_r$ be the reflections in the sides opposite the angles
$p\pi,q\pi,r\pi$, respectively. Then the corresponding triangle group is
$$
(e_1,e_2,e_3)=\Delta(p,q,r)
=
\left\langle
\tau_p,\tau_q,\tau_r
\ \middle|\
\tau_p^2=\tau_q^2=\tau_r^2=1,\;
(\tau_q\tau_r)^{e_1}=(\tau_r\tau_p)^{e_2}=(\tau_p\tau_q)^{e_3}=1
\right\rangle,
$$
where a relation with $e_i=\infty$ is omitted.
\subsection{Hauptmoduln} Let $\Gamma$ be a triangle group in Takeuchi class I. Define
$$
X(\Gamma):=\Gamma\backslash \mathbb{H}^{*},
\;\;\;\;\;
\mathbb{H}^{*}:=\mathbb{H}\cup \Q \cup \{\infty\}.
$$
Then $X(\Gamma)$ has genus $0$. A Hauptmodul $t$ for $\Gamma$ is a modular function
$$
t:X(\Gamma)\longrightarrow \mathbb P^{1}(\mathbb C)
$$
which generates the function field $\mathbb C(X(\Gamma))$. Equivalently, $t$ gives an isomorphism
$$
X(\Gamma)\xrightarrow{\sim}\mathbb P^{1}(\mathbb C),
$$
between algebraic curves and hence is a uniformizer on $X(\Gamma)$.
\par With this definition, the inverse $t$ of the map $D_{abc}$ is precisely a Hauptmodul for the triangle group $(e_1,e_2,e_3)$. A standard choice of the Hauptmoduln appearing in the literature is the Hauptmoduln appearing in~\cite[Appendix]{allen2025explicit}. Our choices are related to those by the natural $S_3$-action on the three singular values $0,1,\infty$, by the cross-ratio transformations
$$
t,\qquad 1-t,\qquad \frac{1}{t},\qquad \frac{1}{1-t},
\qquad \frac{t}{t-1},\qquad \frac{t-1}{t}.
$$
Changing from one Hauptmodul to another by these transformations only permutes the singularities $0,1,\infty$. Table~\ref{table:2} gives, for each triangle group in Takeuchi's class I, the Hauptmodul used in this paper, its relation with the Hauptmodul appearing in~\cite[Appendix]{allen2025explicit}, the corresponding group $\Gamma$, and the preferred vertices $\tau_0=t^{-1}(0)$, $\tau_1=t^{-1}(1)$, and $\infty=t^{-1}(\infty)$ of the Schwarz triangle. Here $u,t_3,t_4^+$ and $t_6^+$ denote the Hauptmoduln used in~\cite[Appendix]{allen2025explicit}.

\begin{center}
\small{
\begin{table}[H]
{
\renewcommand{\arraystretch}{2.85}
\begin{tabular}{ |c|c|c|c|c|c|}
\hline
\multirow{2}{*}{Triangle}  & & & & &  \\[-10pt]
 group & Hauptmoduln $t$ & Comparision & $\tau_0$ & $\tau_1$ & $\Gamma$\\[-4pt]
\hline 
$(\infty, \infty, \infty)$ 
& $\lambda \coloneqq$ Modular $\lambda$ 
& $\lambda$ 
& 1 & 0 & $\Gamma(2)$\\
\hline
$(2, \infty, \infty)$ 
& $S_2 \coloneqq E_{2,2}^2(\tau)\dfrac{\eta^{8}(\tau)}{64\eta^{16}(2\tau)}$ 
& $S_2=\dfrac{u-1}{u}$ 
& $\dfrac{1+i}{2}$ & 0 & $\Gamma_0(2)$ \\
\hline
$(3, \infty, \infty)$ 
& $S_3 \coloneqq 8E_{1,-3}^3(\tau)\dfrac{\eta^{3}(\tau)}{\eta^{9}(3\tau)}$ 
& $S_3=\dfrac{1}{t_3}$ 
& $\dfrac{2+\zeta_3}{3}$ & 0 & $\Gamma_0(3)$  \\
\hline
$(2, 3, \infty)$ 
& $J_3 \coloneqq \dfrac{j}{1728}$ 
& $J_3=\left(\dfrac{1728}{j}\right)^{-1}$ 
& $1+\zeta_3$ & $i$ & $\PSL_2(\Z)$ \\
\hline
$(2, 4, \infty)$ 
& $J_4 \coloneqq \dfrac{S_2^2}{4(S_2-1)}$ 
& $J_4=(t_4^+)^{-1}$ 
& $\dfrac{1+i}{2}$ & $\dfrac{i}{\sqrt{2}}$ & $\Gamma_0^{+}(2)$\\
\hline
$(2, 6, \infty)$ 
& $J_6 \coloneqq \dfrac{S_3^2}{4(S_3-1)}$ 
& $J_6=(t_6^+)^{-1}$ 
& $\dfrac{2+\zeta_3}{3}$ & $\dfrac{i}{\sqrt{3}}$ & $\Gamma_0^{+}(3)$\\
\hline
$(3, 3, \infty)$ 
& $R_3 \coloneqq \dfrac{1 - \sqrt{1-J_3}}{2}$ 
& $4R_3(1-R_3)=J_3$ 
& $1+\zeta_3$ & $\zeta_3$ & $\Big\langle \mathscr{P},\Gamma(2)\Big\rangle$   \\
\hline
$(4, 4, \infty)$ 
& $R_4 \coloneqq \dfrac{1 - \sqrt{1-J_4}}{2}$ 
& $4R_4(1-R_4)=J_4$ 
& $\dfrac{1+i}{2}$ & $\dfrac{-1+i}{2}$ & $\Big\langle \mathscr{Q},\mathscr{R} \Gamma(4)\Big\rangle$ \\
\hline
$(6, 6, \infty)$ 
& $R_6 \coloneqq \dfrac{1 - \sqrt{1-J_6}}{2}$ 
& $4R_6(1-R_6)=J_6$ 
& $\dfrac{2+\zeta_3}{3}$ & $\dfrac{-1+\zeta_3}{3}$ & $\Big\langle  \mathscr{Q} ,   \mathscr{S }\Gamma(6)\Big\rangle$ \\
\hline
\end{tabular}}
\vspace{0.5cm}
\caption{Description of triangle groups and comparison with the Hauptmoduln used in~\cite[Appendix]{allen2025explicit}.}
\label{table:2}
\end{table}}
\end{center}
Here, $\zeta_3$ is the primitive third root of unity $
\zeta_3 = e^{2\pi i/3} = \frac{-1+\sqrt{-3}}{2}
$. $\eta(\tau)$ denotes the Dedekind eta-function
$$
\eta(\tau)= q_\tau^{1/24}\prod_{n=1}^\infty (1-q_\tau^n), \;\;\;\;\; q_\tau=e^{2\pi i \tau}.
$$
$E_{2,2}$ and $E_{1,-3}$ are defined as 
$$E_{2,2}(\tau) = 1+24\sum_{n\geq 1}\Big(\sigma_1(2n)-2\sigma_1(n)\Big)q_\tau^n,\;\;E_{1,-3}(\tau) = \frac{1}{6} + \sum_{n\geq 1} \Big( \sum_{d \mid n}\legendre(-3,d)\Big)q_\tau^n,$$ 
where $\sigma_1(n) = \sum_{d \mid n} d$. They are modular forms of weight $2$, level $2$, and weight $1$, level $3$, respectively. The $2 \times 2$ matrices $\mathscr{P}, \mathscr{Q}, \mathscr{R}, \mathscr{S}$ are as follows:
$$\mathscr{P}=  \bgroup\renewcommand*{\arraystretch}{1} \begin{pmatrix}-1 & -1 \\ 1 & 0 \end{pmatrix} \egroup,\;\; \mathscr{Q} = \bgroup\renewcommand*{\arraystretch}{1} \begin{pmatrix}1 & -2 \\ 0 & 1 \end{pmatrix} \egroup,\;\; \mathscr{R} =  \dfrac{1}{\sqrt{2}}\bgroup\renewcommand*{\arraystretch}{1}\begin{pmatrix}2 & -1 \\ 2 & 0 \end{pmatrix} \egroup, \;\;\mathscr{S} = \dfrac{1}{\sqrt{3}}\bgroup\renewcommand*{\arraystretch}{1} \begin{pmatrix}3 & -1 \\ 3 & 0 \end{pmatrix} \egroup.$$
To specify our choice of Hauptmoduln, first consider the triangle groups $(2,2N,\infty) = \Gamma_0(N)^+=\Gamma_0(N)\cup\Gamma_0(N)W_N$ for $N\in\{2,3\}$ where 
$$\Gamma_0(N)= \Big\{ \begin{pmatrix} a & b \\ c& d \end{pmatrix} \in \SL_2(\Z) \mid c \equiv 0 \bmod \; N\Big\},$$ and $W_N = \frac{1}{\sqrt{N}} \begin{pmatrix}
   0 & -1 \\ N & 0   
\end{pmatrix}$ is the Fricke involution. Set
$$
A_N(\tau)=\left(\frac{\eta(\tau)}{\eta(N\tau)}\right)^{\frac{12}{N-1}},
\;\;\;\text{then}\;\;\;
A_N(W_N\tau)=\frac{N^{\frac{6}{N-1}}}{A_N(\tau)}.
$$
It follows that
$
A_N(\tau)+\frac{N^{\frac{6}{N-1}}}{A_N(\tau)}
$
is invariant under $W_N$, while it changes sign under $\tau\mapsto\tau+1$ for $N=2,3$. Hence its square is invariant under $\Gamma_0(N)^+=\Gamma_0(N)\cup\Gamma_0(N)W_N$.
Normalizing the square so that
$$
J_{2N}(\tau_0)=0,\;\;\;\;\; J_{2N}(\tau_1)=1,\;\;\;\;\; J_{2N}(i\infty)=\infty.
$$
we obtain the Hauptmodul
\begin{equation}\label{haupt}
J_{2N}(\tau)=\frac{1}{4N^{\frac{6}{N-1}}}\left(A_N(\tau)+\frac{N^{\frac{6}{N-1}}}{A_N(\tau)}\right)^2.
\end{equation}
By Ligozat's criterion for eta quotients (see~\cite{ligozat1975courbes}), we can verify that $J_{2N}(\tau)$ is a modular function on $\Gamma_0(N)^{+}$ of level $2N$ and weight $0$. This is the chosen Hauptmodul for the triangle group $(2,2N,\infty)$.
\par For $(d,d,\infty) \leq (2,d,\infty)$, $d=3,4,6$, the quadratic transformation $15.8.20$ in~\cite{DLMF25} gives the corresponding Hauptmoduln $R_d$ such that $J_{2d} = 4R_d(1-R_d)$ with $J_3 = \frac{j}{1728}$. Another quadratic transformation $15.8.14$ in~\cite{DLMF25} provides Hauptmoduln $S_d$ for $(d,\infty,\infty)$, $d=2,3$, with $J_{2d}=\frac{S_d^2}{4(S_d-1)}$. These rational relations between the Hauptmoduln are due to the degree-$2$ coverings between the corresponding modular curves induced by the inclusions $(d,d,\infty)\leq (2,d,\infty)$.
\subsection{Modularity of hypergeometric function}
Special hypergeometric functions often become modular forms under a suitable modular change of variable. Classical identities relating the Gauss hypergeometric function to modular forms appear in the work of Ramanujan. A basic example is the complete elliptic integral of the first kind
$$
K(k)=\int_0^1 \frac{dx}{\sqrt{(1-x^2)(1-k^2x^2)}} = \frac{\pi}{2}\;\;{}_2F_1\!\left(\frac{1}{2},\frac{1}{2},1;k^2\right).
$$
Define
$$
\tau = i\,\frac{K\!\left(\sqrt{1-k^2}\right)}{K(k)}.
$$
Then $k^2$ can be written as a function of $\tau$, namely
$$
k^2(\tau)=\lambda(\tau)=\frac{\theta_2(\tau)^4}{\theta_3(\tau)^4},
\;\;\;
\theta_2(\tau)=\sum_{n\in\mathbb{Z}} q_\tau^{\frac{1}{2}(n+\frac{1}{2})^2}, \;\;\;\;\;
\theta_3(\tau)=\sum_{n\in\mathbb{Z}} q_\tau^{\frac{1}{2}n^2}, \;\;\;\;\; q_\tau=e^{2\pi i \tau}.
$$
In this setting, we have the classical identity
$$
{}_2F_1\!\left(\frac{1}{2},\frac{1}{2},1;\lambda(\tau)\right)=\theta_3(\tau)^2
=
\frac{\eta(\tau)^{10}}{\eta(\tau/2)^4\,\eta(2\tau)^4}.
$$
We can use the Dedekind eta transformation formula
(see~\cite[Theorem 5.8.1]{cohen2017modular}) to obtain
$$
{}_2F_1\!\left(\frac12,\frac12,1;\lambda(\gamma\tau)\right)
=
\left(\frac{-4}{d}\right)(c\tau+d)\,
{}_2F_1\!\left(\frac12,\frac12,1;\lambda(\tau)\right)\;\;\text{for}\;\;\gamma=
\begin{pmatrix}
a&b\\ c&d
\end{pmatrix} \in \Gamma_0(4).
$$
This expresses the hypergeometric function as a weight-1 modular form for $\Gamma_0(4)$ with multiplier $\left(\frac{-4}{d}\right)$. See~\cite[section 2.4]{li2018computing} for details. 
\par This example is particularly simple in that the hypergeometric function itself
becomes a modular form. In general, however, a single local solution
\({}_2F_1(a,b;c;t(\tau))\) need not be modular. As shown by
Yang~\cite[Theorem~9]{MR3011876}, automorphic forms on a triangle group are obtained from suitable algebraic factors and linear combinations of the two
local hypergeometric solutions. We now make this precise and apply it later in the setting of arithmetic triangle groups in Takeuchi’s class I. 
\par For the identities in the following theorem, the branch of the $_2F_1$ is fixed as follows. We choose the Schwarz triangle corresponding to the datum $(a,b,c)$ and the Hauptmodul $t(\tau)$ with $t(\tau_0)=0$. Near $\tau_0$, the function $t(\tau)$ maps a neighborhood of $\tau_0$ to a neighborhood of $z=0$, and the expression $_2F_1(a,b,c;t(\tau))$
is initially defined using the power series branch at $z=0$. We then continue this branch along paths inside the chosen Schwarz triangle. Since the Schwarz triangle is simply connected and does not contain the singular values $0,1,\infty$ except at its vertices, this determines a single-valued branch on the interior of the triangle. This is the branch used throughout the hypergeometric--modular identities. Equivalently, it is the branch satisfying the normalized value $_2F_1(a,b,c;t(\tau_0))=1$.

\begin{theorem}\label{thm:hgmodid}
Let $e_1, e_2, e_3 \in \N \cup \{\infty\}$. For a triangle group $\Gamma(e_1,e_2,e_3)$, 
choose $a,b,c$ such that 
$$
\{|1-c|, |c-a-b|, |a-b|\} = \left\{\frac{1}{e_1}, \frac{1}{e_2}, \frac{1}{e_3}\right\}.
$$
Let $t$ be a Hauptmodul for $\Gamma =\Gamma(e_1,e_2,e_3)$
and let $0 \neq \tau_0=t^{-1}(0)$. Define 
$$
\mathcal{W}(\tau)=t'(\tau)t(\tau)^{-c}(1-t(\tau))^{c-a-b-1},
\;\;\;\;t'(\tau)=\frac{d t(\tau)}{d \tau},$$
and let \(\mathcal W(\tau_0)\) denotes the nonzero limiting value of the chosen local branch of
\(\mathcal W(\tau)\) as \(\tau\to\tau_0\), and the branch of the square root is chosen so that
\[
\left(\frac{\mathcal W(\tau)}{\mathcal W(\tau_0)}\right)^{1/2}
\to 1
\qquad\text{as}\qquad \tau\to\tau_0.
\] 
Then for $\tau$ in a sufficiently small simply connected neighborhood of $\tau_0$,
\begin{equation}\label{eq:hgmodid}
 _2F_1\begin{bmatrix}
\begin{matrix} a \\ \\ \end{matrix} & \begin{matrix} b \\ c \end{matrix} & ; \: t(\tau)
\end{bmatrix}
=
\dfrac{\tau+\alpha \tau_0}{\tau_0+\alpha \tau_0}
\Big(\dfrac{\mathcal{W}(\tau)}{\mathcal{W}(\tau_0)}\Big)^{\frac{1}{2}},
\end{equation}
where $\alpha \in \C$ is an explicit complex number. If $\Gamma(e_1,e_2,e_3)$ is a triangle group in Takeuchi Class I, $\alpha$ becomes a root of unity as given in Table \ref{table:identity}. 
\end{theorem}
\begin{proof}
For a given $(a,b,c)$, we recall $\text{HDE}(a,b,c;z)$~\eqref{eq00} satisfied by $_2F_1(a,b,c;z)$
\begin{equation}\label{hde}
y''(z) + P(z)y'(z) + Q(z)y(z) = 0,\;\;P(z)=\frac{c-(a+b+1)z}{z(1-z)},\;\;Q(z)=-\frac{ab}{z(1-z)}.
\end{equation}
The Wronskian of this differential equation is
$$
\mathcal{W}(z) = \lambda e^{-\int P(z)dz} = z^{-c}(1-z)^{c-a-b-1},
$$ for some constant $\lambda \in \C^{\times}$.
Take $z=t(\tau)$. Then, due to the chain rule
$$
\mathcal{W}(\tau) = \lambda t'(\tau) t(\tau)^{-c} (1-t(\tau))^{c-a-b-1}.
$$
The normal form of the above differential equation
$$
u''(z)+R(z)u(z)=0,\;\;\text{where}\;\;R(z)= Q(z)-\frac{P'(z)}{2}-\frac{P^2(z)}{4}
$$
has a solution
$$
u(z) = 
u(t(\tau))
=
y(t(\tau))
\left(\frac{t'(\tau)}{\mathcal W(\tau)}\right)^{1/2}.
$$
Viewing $\tau$ as a function of $z$, the Schwarzian derivative $\{\tau,z\}$ of $\tau$ with respect to $z$ defined as
$$
\{\tau,z\} = \frac{\tau'''(z)}{\tau'(z)}
-\frac{3}{2}\left(\frac{\tau''(z)}{\tau'(z)}\right)^2\;\;\;\text{satisfies}\;\;\; \{\tau,z\} = 2R(z).
$$
A final substitution
$$
v(\tau) = u(t(\tau))(t'(\tau))^{-\frac{1}{2}}
$$
gives
$$
v''(\tau)+\Big(R(t(\tau)) t'(\tau)^2 + \frac{1}{2} \{t,\tau\}\Big)v(\tau) = v''(\tau) = 0.
$$
since  the Schwarzian derivative,  $\{t,\tau\}=-2R(t(\tau))t'(\tau)^2$, follows from the chain rule. So the solution to this equation is $v(\tau)=c_1\tau+c_2$ for some complex constants $c_1$ and $c_2$. Substituting back, we obtain
$$
y(t(\tau)) = (c_1\tau+c_2)\mathcal{W}(\tau)^{\frac{1}{2}}.
$$
We first show that \(c_1\neq0\).  Let \(\gamma_0\in\Gamma\) be a
nontrivial stabilizer of the elliptic point \(\tau_0\). Since
\(t(\gamma_0\tau)=t(\tau)\), the left-hand side in \eqref{eq:hgmodid} is invariant under
\(\gamma_0\). On the other hand, since \(t\) is invariant under \(\Gamma\), we have
\[
t(\gamma_0\tau)=t(\tau),
\;\;\;\;
t'(\gamma_0\tau)
=
(c_0\tau+d_0)^2t'(\tau),
\qquad \text{for all}\;\;\;
\gamma_0=
\begin{pmatrix}
a_0&b_0\\ c_0&d_0
\end{pmatrix} \in \Gamma.
\]
Hence
\[
\mathcal W(\gamma_0\tau)
=
(c_0\tau+d_0)^2\mathcal W(\tau),\;\;\;\;
\mathcal W(\gamma_0\tau)^{1/2}
=
\varepsilon(c_0\tau+d_0)\mathcal W(\tau)^{1/2},
\]
for some constant \(\varepsilon\in\C^\times\), depending only on the choice of square-root branch.
Since \(\gamma_0\) is elliptic with
fixed point \(\tau_0\in\mathbb H\), we have \(c_0\neq 0\), and therefore
\(c_0\tau+d_0\) is not constant. Thus the linear factor \(c_1\tau+c_2\) cannot be
constant; otherwise it could not cancel the nonconstant automorphy factor of
\(\mathcal W(\tau)^{1/2}\). Hence \(c_1\neq0\).
\newline Near $t = 0$, we consider the solution $_2F_1(a,b,c;t(\tau))$ of \eqref{hde}. Since $_2F_1(a,b,c;0)=1$ and \(t(\tau_0)=0\), after normalization we obtain
\begin{equation}\label{eq2F1}
 _2F_1\begin{bmatrix}
\begin{matrix} a \\ \\ \end{matrix} & \begin{matrix} b \\ c \end{matrix} & ; \: t(\tau)
\end{bmatrix}
=
\dfrac{c_1\tau+c_2}{c_1\tau_0+c_2}
\Big(\dfrac{\mathcal{W}(\tau)}{\mathcal{W}(\tau_0)}\Big)^{\frac{1}{2}}.
\end{equation}
Therefore, since \(c_1\neq0\) and \(\tau_0\neq0\), we may write
$c_1\tau+c_2=c_1\left(\tau+\alpha\tau_0\right)$ for $\alpha=\frac{c_2}{c_1\tau_0}$.
Then
\[
\frac{c_1\tau+c_2}{c_1\tau_0+c_2}
=
\frac{\tau+\alpha\tau_0}{\tau_0+\alpha\tau_0}.
\]
For the Class I cases considered below, the values of \(\alpha\) are roots of unity. They are computed
from the corresponding stabilizer relation and are listed in
Table~\ref{table:identity}.
\end{proof}
Table~\ref{table:identity} gives the resulting hypergeometric-modular identities for the class I cases considered here, including the chosen Hauptmodul, the point $\tau_0$, the constant $\alpha$, and the corresponding square-root Wronskian factor.  
\begin{center}
\begin{table}[H]
{ 
\renewcommand{\arraystretch}{1.8} 
\begin{tabular}{ |c|c|c|c|c|c| }
\hline 
Triangle group & $(a,b,c)$ & $t$& $\tau_0$ & $\alpha$ & $\mathcal{W}(\tau)^{\frac{1}{2}}$ \\
\hline
$(2,\infty,\infty)$ & $(\frac{1}{4}, \frac{1}{4}, \frac{1}{2})$ & $S_2$ & $\frac{1+i}{2}$& $i$ & $\frac{\eta^4(2\tau)}{\eta^2(\tau)}$ \\
\hline 
$(3,\infty,\infty)$ & $(\frac{1}{3}, \frac{1}{3}, \frac{2}{3})$ & $S_3$ & $\frac{3+\sqrt{-3}}{6}$& $ \zeta_6$ & $\frac{\eta^3(3\tau)}{\eta(\tau)}$ \\
\hline
$(3,3,\infty)$ & $(\frac{1}{6}, \frac{1}{6}, \frac{2}{3})$ & $R_3$ & $-\bar{\zeta_3}$& $i$ & $\eta^2(\tau)$ \\
\hline 

$(4,4,\infty)$ & $(\frac{1}{4}, \frac{1}{4}, \frac{3}{4})$ & $R_4$ & $\frac{1+i}{2}$ & $i$ & $\eta(\tau)\eta(2\tau)$ \\
\hline
$(6,6,\infty)$ & $(\frac{1}{3}, \frac{1}{3}, \frac{5}{6})$ & $R_6$ & $\frac{3+\sqrt{-3}}{6}$ &  $\zeta_3$   & $\eta(\tau)\eta(3\tau)$ \\
\hline
\multirow{2}{*}{$(2, 4, \infty)$} 
&  $(\frac{1}{8}, \frac{1}{8}, \frac{1}{2})$ & $1-J_4$& $\frac{i}{\sqrt{2}}$ & 1 & $\eta(\tau)\eta(2\tau)$ \\
&  $(\frac{1}{8}, \frac{1}{8}, \frac{3}{4})$ & $J_4$& $\frac{1+i}{2}$ & $i$ & $\eta(\tau)\eta(2\tau)$ \\
\hline
\multirow{2}{*}{$(2, 6, \infty)$}
&  $(\frac{1}{6}, \frac{1}{6}, \frac{1}{2})$ & $1-J_6$& $\frac i{\sqrt 3}$ & 1 & $\eta(\tau)\eta(3\tau)$ \\
 &  $(\frac{1}{6}, \frac{1}{6}, \frac{5}{6})$ & $J_6$& $\frac{3+\sqrt{-3}}{6}$ & $\zeta_3$  & $\eta(\tau)\eta(3\tau)$ \\
\hline 
\end{tabular}
}
\vspace{0.5cm}
\caption{Hypergeometric Modular Identities for Class I}
\label{table:identity}
\end{table}
\end{center}
\begin{example}
We illustrate the computation of the square-root Wronskian factor and the constant $\alpha$ corresponding to the first row of the above table where the triangle group is $\Gamma_0(2) = (2,\infty,\infty)$. Choose $(a,b,c)=\left(\frac14,\frac14,\frac12\right)$ such that 
$$\{|1-c|, |c-a-b|, |a-b|\} = \left\{\frac{1}{2}, 0, 0\right\}.$$
The Hauptmodul is chosen from Table~\ref{table:2} as follows.
$$
t(\tau)=S_2(\tau), \;\;\;\;\;\; \tau_0=\frac{1+i}{2}.
$$
For this datum, the pulled-back Wronskian is
$$
\mathcal{W}(\tau)
=
S_2'(\tau)\,S_2(\tau)^{-1/2}(1-S_2(\tau))^{-1}.
$$
We can verify that $\mathcal W(\tau)$ is an modular form of weight $2$ on $\Gamma_0(4)$.
By~\cite[Section 2.4]{li2018computing} and~\cite[Appendix]{allen2025explicit}, we obtain
$$
\mathcal W(\tau)=\pi i\,\theta_2(\tau)^4\;\;\;\text{where}\;\;\;
\theta_2(\tau)=2\frac{\eta(2\tau)^2}{\eta(\tau)}.
$$
So, if we fix a choice of the branch of the square root,
$$
\mathcal W(\tau)^{1/2}
=
(16\pi i)^{1/2}\frac{\eta(2\tau)^4}{\eta(\tau)^2}.
$$
We now compute $c_1$ and $c_2$. Since $\tau_0=\frac{1+i}{2}$ is an elliptic point of $\Gamma_0(2)$, its stabilizer can be chosen as
$$
\gamma_0=
\begin{pmatrix}
1 & -1\\
2 & -1
\end{pmatrix},
\;\;\;\;\;\; \text{as }\quad 
\gamma_0(\tau_0)=\tau_0.
$$
Then we have
$$
{}_2F_1\!\left(\frac14,\frac14,\frac12;S_2(\gamma_0\tau_0)\right)
=
{}_2F_1\!\left(\frac14,\frac14,\frac12;S_2(\tau_0)\right).
$$
Using \eqref{eq2F1} write
\begin{equation}\label{eqexample}
{}_2F_1\!\left(\frac14,\frac14,\frac12;S_2(\tau)\right)
=
\frac{c_1\tau+c_2}{c_1\tau_0+c_2}
\frac{\eta(2\tau)^4/\eta(\tau)^2}
{\eta(2\tau_0)^4/\eta(\tau_0)^2}.
\end{equation}
Since the denominator is independent of $\tau$, it does not affect the transformation relation used to determine $c_1$ and $c_2$.
The invariance of the left-hand side given in~\eqref{eqexample} under $\gamma_0$ and the following transformation of the eta quotient under $\gamma_0$ in the right side give a relation between the linear factor $c_1\tau+c_2$.
Since, $$
\eta(\gamma_0\tau)=e^{\pi i/4}(2\tau-1)^{1/2}\eta(\tau),
\;\;\;
\text{and}
\;\;\;
\eta(2\gamma_0\tau)=-i(2\tau-1)^{1/2}\eta(2\tau).
$$
we have
$$
\frac{\eta(2\gamma_0\tau)^4}{\eta(\gamma_0\tau)^2}
=
-i(2\tau-1)\frac{\eta(2\tau)^4}{\eta(\tau)^2}.
$$
Substituting this into the invariance relation gives
$
(c_1\gamma_0\tau+c_2)(-i)(2\tau-1)=c_1\tau+c_2.
$
We now have
$$
-i\bigl(c_1(\tau-1)+c_2(2\tau-1)\bigr)=c_1\tau+c_2\;\;\;\text{since}\;\;\;
\gamma_0\tau=\frac{\tau-1}{2\tau-1}.
$$
Comparing coefficients of $\tau$ and the constant terms, we obtain
$
-i(c_1+2c_2)=c_1$, and $i(c_1+c_2)=c_2.
$
Solving this system gives
$
c_1=2$, and $c_2=i-1,
$
up to a common nonzero scalar, which does not affect the normalized ratio
$
\frac{c_1\tau+c_2}{c_1\tau_0+c_2}.
$
Thus, for the first row of Table~\ref{table:identity}, we obtain
$$
\mathcal{W}(\tau)^{1/2}=\frac{\eta(2\tau)^4}{\eta(\tau)^2},
\;\;\;\;\;\;
c_1=2,
\;\;\;\;\;\;
c_2=i-1,
\;\;\;\;\;\;\alpha=\frac{c_2}{c_1 \tau_0} = i.$$
$$
{}_2F_1\!\left(\frac14,\frac14,\frac12;S_2(\tau)\right)
=
\frac{2\tau+i-1}{2i}
\frac{\eta(2\tau)^4/\eta(\tau)^2}
{\eta(2\tau_0)^4/\eta(\tau_0)^2} = \frac{\tau + i\tau_0}{\tau_0 + i\tau_0}
\frac{\eta(2\tau)^4/\eta(\tau)^2}
{\eta(2\tau_0)^4/\eta(\tau_0)^2}.
$$
\end{example}
\section{Exceptional Sets for Takeuchi's Class I} \label{sec:ex1}
We now determine the exceptional sets $E(a,b,c)$ for all data $(a,b,c)$ corresponding to arithmetic triangle groups in Takeuchi’s class I. The main idea is to combine the modular identities obtained in the previous section with results on the algebraicity of periods. 
\par We expressed the hypergeometric function in terms of modular functions evaluated at points in the upper half-plane. This allows us to reduce the problem to studying the algebraicity of certain modular functions at CM points. The forward direction follows from classical results on modular functions, while the converse requires transcendence results for periods of abelian varieties.
\begin{lemma}[Schneider \cite{Schneider1937}] \label{l2}
Let $\tau \in \overline{\Q} \cap \UH$, then $t(\tau) \in \overline{\Q}$ if $\tau$ is a CM point, for any Hauptmodul $t$ in Table~\ref{table:2}. 
\end{lemma}
\begin{proof} 
    The results follow form the algebraic relation of the elliptic $j$-function and the Hauptmodul in Table~\ref{table:2}, and a classical result of Schneider that  for any $\tau \in \overline{\Q}$, if $\tau$ is a CM point then $j(\tau) \in \overline{\Q}$. See also \cite[Proposition 5.10.3, Corollary 5.10.4]{cohen2017modular} and~\cite[Theorem 2.6.1]{silverman1994advanced}. 
\end{proof}
This lemma shows that the chosen Hauptmodul takes algebraic values at CM points. Using the modular identities obtained in Section~\ref{sec:hypergeom}, this implies that the corresponding values of the hypergeometric function are algebraic. This yields the forward inclusion in the description of the exceptional set.

\begin{theorem} \label{tforward}
If $z = t(\tau_d)$ for $t$ and $\tau_d \in \Q(\sqrt{-d}) \cap \UH$ as in the Table~\ref{table:main}, then $z \in E(a,b,c)$. 
\end{theorem}
\begin{proof} Since $\tau_d$ is a CM point, $z$ is algebraic by Lemma~\ref{l2}. Now, let $\tau_0 = t^{-1}(0)$. Then $\tau_0\in  \Q(\sqrt{-d})$ and thus
$$ _2F_1\begin{bmatrix}
\begin{matrix} a \\ \\ \end{matrix} & \begin{matrix} b \\ c \end{matrix} & ; \: t(\tau_d) 
\end{bmatrix}= \dfrac{c_1\tau_d+c_2}{c_1\tau_0+c_2} \Big(\dfrac{\mathcal{W}(\tau_d)}{\mathcal{W}(\tau_0)}\Big)^{\frac{1}{2}}$$ is algebraic,  
which follows from the algebraicity of $\frac{\eta(\frac{\alpha\tau + \beta}{\delta})}{\eta(\tau)}$ for $\alpha,\beta,\delta \in \Z$ with $\alpha\delta >0$,~\cite[ Theorem 12.2.2]{lang1987elliptic}. 
\end{proof}
\begin{example} Consider the datum $(\frac{1}{3}, \frac{1}{3}, \frac{5}{6})$ associated to the triangle group $(6,6,\infty)$. For any $\tau_d \in \Q(\sqrt{-d}) \cap \UH$ where $d= 3$, $\tau_d$ can be expressed as $\frac{\alpha\sqrt{-3}+\beta}{\delta}$ with $\alpha,\beta,\delta \in \Z$ and $\alpha\delta >0$. For all such $\tau_d$, $\eta(\tau_d) \in \overline{\Q} \cdot \eta(\sqrt{-3})$ and hence,
$$\dfrac{\eta(\tau_d) \eta(3\tau_d)}{\eta(\frac{3+\sqrt{-3}}{6}) \eta(\frac{3+\sqrt{-3}}{2})} \in \overline{\Q}.$$ 
\end{example}
\subsection*{Proof of Theorem~\ref{tmain}}
Due to \eqref{euler}, we have
$$
\hyp(a,b,c,z) = \dfrac{\int_0^1 \omega}{B(b,c-b)}, \quad \mbox{where } \omega= x^{b-1} (1-x)^{c-b-1} (1-zx)^{-a} dx. 
$$
The algebraicity of $_2F_1(a,b,c;z)$ depends on the transcendence of the quotient of periods. Wolfart and W\"ustholz proved algebraic relations between the periods on simple abelian varieties. 
\begin{theorem}[W\"ustholz's Theorem~\cite{wolfart1985uberlagerungsradius}] \label{theoremwustholz}
Let $T_1, T_2, \dots, T_n$ be pairwise nonisogenous simple abelian varieties defined over $\overline{\Q}$. Let $k_j$ be the dimension of $T_j$ for $1 \leq j \leq n$. Suppose $\{\omega^{(j)}_k\}_{1\le k\le k_j}$ and $\{\eta^{(j)}_k\}_{1\le k\le k_j}$ be the bases of differential forms on $T_j$ of first and second kind respectively. Then, for any arbitrary choice of non-zero $\gamma_j \in H_1(T_j, \Z)$, 
$$
\dim_{\overline{\Q}}\!\Big\{2\pi i,\; \int_{\gamma_j}\omega^{(j)}_k,\; \int_{\gamma_j}\eta^{(j)}_k \;\Big|\; 1\le j\le n,\; 1\le k\le k_j\Big\}
= 1 + 2\sum_{j=1}^n k_j.
$$
\end{theorem}
To apply this result, it is necessary to determine whether the differential forms appearing in Euler's integral representation  \eqref{eq1}, give rise to periods of the first or second kind. We investigate if they satisfy inequalities of the form \eqref{eq1} and \eqref{eq2}.  Throughout, for a given $(a,b,c)$, we let $M=\mbox{L.C.D}(b,c)$. $N=\mbox{L.C.D}(a,b,c)$, and define  
$$
     T_z=\text{Jac}_{\text{new}}(X(N,z))\quad \mbox{and} \quad  T_0=\text{Jac}_{\text{new}}(X(M,0)).
$$
The new Jacobians are defined in Section~\ref{sec:basic}; see~\eqref{eqJac1}. Then $T_z$ and $T_0$ are abelian varieties defined over $\overline{\Q}$.
\begin{lemma}\label{lemma:period kind}
For all the $(a,b,c)$ as of Table~\ref{table:0} and Table~\ref{table:main}, the integral $\int_0^1 x^{b-1} (1-x)^{c-b-1} (1-zx)^{-a} dx$ is a period of first kind on $T_z$. 
If $c \neq 1$, the beta integral $B(b,c-b)$ is a period on $T_0$. In particular,
$$
B(b,c-b) \mbox{ is } 
\begin{cases}
\text{a period of first kind,} &\text{if}\; c<1, \\[6pt]
\text{a period of second kind, } &\text{if}\; c>1, \\[6pt]
\text{an algebraic multiple of }\; \pi,  &\text{if}\; c=1.  
\end{cases}
$$
\end{lemma}

The transcendence result implies that the algebraicity of the hypergeometric value imposes constraints on the associated period data. In particular, it forces the corresponding values of the Schwarz triangle function   $D_{abc}(z)$ to lie in an imaginary quadratic field. The following lemma determines the CM field based on $a,b,c$.
\begin{lemma} \label{lemma10}
Let $(a,b,c)$ be any datum in  Table~\ref{table:main}. If $z \in E(a,b,c)$, any Schwarz map $D_{abc}(z)$  takes value in $\Q(\sqrt{-d})$. 
\end{lemma}
\begin{proof}
Let $\omega = x^{b-1} (1-x)^{c-b-1} (1-zx)^{-a} dx$. In this case, by Lemma \ref{lemma:period kind}, the period $\int_0^1\omega$ is of first kind on $T_z$ and $B(b,c-b)$ is of first kind on $T_0$.
Since $_2F_1(a,b,c;z)=\frac{\int_0^1\omega}{B(b,c-b)} \in \overline{\Q}$, W\"ustholz's Theorem implies that there are simple abelian isogenous subvarieties $T^{\prime}_z$ and $T^{\prime}_0$ of $T_z$ and $T_0$ respectively. 
\newline For all $(a,b,c)$ in Table~\ref{table:main}, $M \in \{3,4,6,8\}$. If $M \neq 8$, the dimension of $T_0$ is $\frac{\phi(M)}{2} =1$. So, we have $T^{\prime}_z \sim T^{\prime}_0 \sim T_0$. Due to the action of $\mu_M$ on $T_0$, the endomorphism ring of $T_0$ contains $Z[\zeta_M]$. 
Hence, $T_0 \sim \C/\Lambda(0)$ where $\Lambda(0) = \Z \oplus \tau_M\Z$ with $\tau_M \in \Q(\zeta_M)$. For $M \in \{3,4,6\}$, the corresponding integers $d$ are given in the following table. 
\\
\begin{center}
\renewcommand{\arraystretch}{1.5}
\begin{tabular}{|w{c}{2cm}|w{c}{2cm}|w{c}{2cm}|w{c}{2cm}|}
\hline
$M$ & $\Q(\zeta_M)$ & $d$ & $\Q(\sqrt{-d})$ \\
\hline
$3$ & $\Q(\zeta_3)$ & $3$ & $\Q(\sqrt{3}i)$ \\
\hline
$4$ & $\Q(\zeta_4)$ & $1$ & $\Q(i)$ \\
\hline
$6$ & $\Q(\zeta_6)$ & $3$ & $\Q(\sqrt{3}i)$ \\
\hline
\end{tabular}
\end{center}
\vspace{0.6cm}
By definition, $T_z \sim \C^{\phi(N)}/\Lambda(z)$. We can express $\Lambda(z)$ in terms of Schwarz function as
$$
\Lambda(z)\coloneqq
\left\{\Bigl(
  \sigma_i(u)\,
  \;+\;
  \sigma_i(v)\, D_{abc}(z)
\Bigr) : (i,N) = 1, \omega_i \;\text{is basis element of}\;V_i \; ; \; u,v\in \Z[\zeta_N] \right\}
$$
up to a permutation of $u$ and $v$ in each of the summand. Moreover $z \in E(a,b,c)$ implies that both $z$ and $_2F_1(a,b,c;z)$ are algebraic. Since, $T_z$ and $T_0$ are defined over $\overline{\Q}$, there is a $\overline{\Q}$-linear transformation $\lambda : \C \rightarrow \C^{\phi(N)}$ with the property $\lambda(\Lambda(0)) \subseteq \Lambda(z)$. Due to the linearity of $\lambda$ applied to $\tau_M$, we have $D_{abc}(z) \in \Q(\zeta_M) = \Q(\sqrt{-d})$. 

For the $(a,b,c)$ corresponding to the triangle group $(2,4,\infty)$, $T_0$ is a variety of dimension $2$, since $M=8$. In the case of $B(\frac{1}{8},\frac{5}{8})$ corresponding to $(a,b,c) = (\frac{1}{8}, \frac{1}{8}, \frac{3}{4})$, we apply multiplication and reflection formula to obtain
$$B(\frac{1}{8},\frac{5}{8}) \in \overline{\Q} \cdot B(\frac{1}{2},\frac{1}{4}).$$
Since $B(\frac{1}{2},\frac{5}{4})$ is a period of the first kind on a variety of dimension $1$, W\"ustholz's Theorem (\ref{theoremwustholz}) applies here, and we get $D_{abc}(z) \in \Q(\sqrt{-1})$. 
\par For $B(\tfrac18,\tfrac38)$ corresponding to $(a,b,c) = (\tfrac18,\tfrac18,\tfrac12)$, we determine
the isogeny type of $T_0$ precisely. For that, we need the following lemma.
\begin{lemma}[Milne]~\cite[Ch.~I, \S3, Prop.~3.13]{milne2020cm} \label{lemma:milne-imprimitive}
Let $K$ be a CM field, Galois over $\Q$, with $\mathrm{Gal}(K/\Q)$ identified
with a subgroup of $(\Z/M\Z)^\times$ via $K=\Q(\zeta_M)$. Let $\Phi \subset
\mathrm{Gal}(K/\Q)$ be a CM-type, and let $H = \{\sigma \in \mathrm{Gal}(K/\Q) :
\sigma\Phi = \Phi\}$. If $H \neq \{1\}$, then $\Phi$ is imprimitive, extended
from the fixed field $K_0 = K^H$, and any abelian variety $A$ of CM-type
$(K,\Phi)$ is isogenous to $B^{[K:K_0]}$ for some abelian variety $B$ with CM
by $K_0$.
\end{lemma}
For $B(\tfrac18,\tfrac38)$ corresponding to $(a,b,c) = (\tfrac18,\tfrac18,\tfrac12)$, we determine
the isogeny type of $T_0$ precisely. Here $M=8$, $P=7$, $Q=5$, and
$C(8,0): y^8 = x^7(1-x)^5$. Applying~\eqref{eq2} to $n=1,3,5,7$ gives
$\dim V_1 = \dim V_3 = 1$ and $\dim V_5 = \dim V_7 = 0$, so $T_0$ has CM-type
$(\Q(\zeta_8), \Phi)$ with $\Phi = \{1,3\} \subset (\Z/8\Z)^\times$. Since $\Phi$
coincides with the order-$2$ subgroup $H=\{1,3\}$, Lemma~\ref{lemma:milne-imprimitive}
applies.
\par Define $\sigma_3$ on $\Q(\zeta_8)$ by $\zeta_8 \mapsto \zeta_8^3$. Since
$\sigma_3(i)=-i$ and $\sigma_3(\sqrt2)=-\sqrt2$, it fixes $\sqrt{-2}$ in
$\Q(\zeta_8)=\Q(i,\sqrt2)$, and since $[\Q(\zeta_8):\Q(\sqrt{-2})]=2=|H|$, we get
$K_0=\Q(\sqrt{-2})$. By Lemma~\ref{lemma:milne-imprimitive}, $T_0$ is isogenous
to $B^2$ for some elliptic curve $B$ with CM by $K_0$. Any two elliptic curves
with CM by the same imaginary quadratic field are isogenous, so $B$ is
isogenous to $T^0 = \C/(\Z\oplus\sqrt2\,i\,\Z)$, which has CM by
$\Z[\sqrt{-2}]\subset K_0$. So $T_0 \sim T^0\times T^0$ is not simple, and the
simple $T_0'$ from Theorem~\ref{theoremwustholz} must be isogenous to $T^0$. We
conclude $D_{abc}(z) \in \Q(\sqrt{-2})$.
In this case, $T^{\prime}_0$ can either be isogenous to $T_0$ or $T^0$, since these are the only possible subvarieties of $T_0$.
We again apply W\"ustholz's Theorem in this case to obtain that $D_{abc}(z) \in \Q(\sqrt{-2})$. For $M=8$, the corresponding integers are given in the following table.
\begin{center}
\renewcommand{\arraystretch}{1.5}
\begin{tabular}{|w{c}{2cm}|w{c}{2cm}|w{c}{2cm}|}
\hline
$(a,b,c)$ & $d$ & $\Q(\sqrt{-d})$ \\
\hline
$(\frac{1}{8}, \frac{1}{8}, \frac{3}{4})$ & $1$ & $\Q(i)$ \\
\hline
$(\frac{1}{8}, \frac{1}{8}, \frac{1}{2})$ & $2$ & $\Q(\sqrt{2}i)$ \\
\hline
\end{tabular}
\end{center}
This completes the proof. 
\end{proof}
To prove Theorem~\ref{tmain}, it remains prove the converse of Theorem~\ref{tforward}.
\begin{theorem}
If $z \in E(a,b,c)$, then $z = t(\tau_d)$ for $t$ and $\tau_d$ as in the Table~\ref{table:main}.
\end{theorem}
\begin{proof}
In Table~\ref{table:main}, for each $(a,b,c)$, the Hauptmodul $t$ has been chosen in such a way that $t$ is the inverse of the Schwarz triangle function $D_{abc}$. For $z \in E(a,b,c)$, due to Lemma~\ref{lemma10}, let $D_{abc}(z) = \tau_d \in \Q(\sqrt{-d})$. So, $z = t (\tau_d)$. 
This completes the proof of Theorem~\ref{tmain}.
\end{proof} 
\begin{example}\label{appendixexample} 
We take a particular example for the triangle group $(2,\infty,\infty)$ that 
$$
9\in E\!\left(\frac14,\frac14,\frac12\right),
$$
for which the corresponding Hauptmodul is $S_2$. Take the CM point $\tau=i$. A direct computation gives that $S_2(i)=9$ and
$$
{}_2F_1\!\left(\frac14,\frac14,\frac12;S_2(i)\right)
=
{}_2F_1\!\left(\frac14,\frac14,\frac12;9\right)
=
\frac{2-i}{2\sqrt2}.
$$
We now illustrate the procedure of finding these special values. More such values are listed in Table~\ref{table:explicit_examples}.

Using the eta transformation formula at $\tau=2i$ gives
\[
\eta\!\left(\frac{i}{2}\right)=\sqrt2\,\eta(2i). 
\]
We evaluate the eta quotient from the eta-product formula for the modular lambda function,
\[
\lambda(\tau)
=
16\,\frac{\eta(\tau/2)^8\eta(2\tau)^{16}}{\eta(\tau)^{24}},\;\;\;
\lambda(i)
=
16\,\frac{\eta(i/2)^8\eta(2i)^{16}}{\eta(i)^{24}}
=
16\cdot 2^4
\left(\frac{\eta(2i)}{\eta(i)}\right)^{24}= -4u(i),
\]
where $u(\tau)=-64\eta(2\tau)^{24}/\eta(\tau)^{24}$ in Table \ref{table:2}. Since \(i\) is fixed by \(\tau\mapsto -1/\tau\), and
\(\lambda(-1/\tau)=1-\lambda(\tau)\), we have $\lambda(i)=\frac12$ and $u(i)=-1/8$.
Thus
\[
S_2(i)=\frac{u(i)-1}{u(i)}=9.
\]
Using the eta transformation formulas, we find
\[
\frac{\eta(2i)^4}{\eta(i)^2}
=
\frac{\eta(i)^2}{2\sqrt2},
\qquad
\frac{\eta(1+i)^4}
{\eta\!\left(\frac{1+i}{2}\right)^2}
=
\frac{1+i}{2}\eta(i)^2.
\]
The first identity follows from the above eta-product formula for \(\lambda\). For the second identity, we use the transformation \(\tau\mapsto -1/\tau\) for \( \tau = (1+i)/2\). Thus \(\eta((1+i)/2)\) is reduced to \(\eta(i-1)=e^{-\pi i/12}\eta(i)\). Substituting these into the hypergeometric--modular identity for $\tau_0=\frac{1+i}{2}$ yields
\[
{}_2F_1\!\left(\frac14,\frac14,\frac12;9\right)
=
\frac{2\tau+i-1}{2\tau_0+i-1}
\frac{\eta(2\tau)^4/\eta(\tau)^2}
{\eta(2\tau_0)^4/\eta(\tau_0)^2} = \frac{3i-1}{2i}\cdot \frac{1}{2\sqrt2(1+i)} =
\frac{2-i}{2\sqrt2}.
\]

\end{example}
This completes the determination of the exceptional set for all data in Takeuchi’s class I for which non-trivial algebraic values occur. But there are cases for which the exceptional set can be trivial. Table~\ref {table:0} records the data for which the exceptional set is trivial. 

\begin{theorem} \label{thm:transcendental}
For all $(a,b,c)$ in Table~\ref{table:0}, $E(a,b,c)=\{0\}$.
\end{theorem}
\begin{proof} For all $(a,b,c)$ in Table~\ref{table:0}, we have $c \geq1$. If $c >1$, the inequality \eqref{eq1} is satisfied, but \eqref{eq2} is not. So, $_2F_1(a,b,c;z)=\frac{\int_0^1 \omega}{B(b,c-b)}$ where $\int_0^1\omega$ and $B(b,c-b)$ are periods of first and second kind on abelian varieties $T_z$ and $T_0$ respectively. By W\"ustholz's theorem, for $z\neq 0$ the quotient is not algebraic since $\int_0^1 \omega$ and $B(b,c-b)$ are $\overline{\Q}$-linearly independent. 
\newline For $c=1$,  $\int_0^1\omega$ is a period of first kind as above but $B(b,c-b) \in \pi \cdot \overline{\Q}$. Again, W\"ustholz's theorem implies that for $z\neq 0$, $_2F_1(a,b,c;z)$ is not algebraic. Since $_2F_1(a,b,c;0)=1$, it follows that $E(a,b,c)=\{0\}$.
\end{proof}
\begin{center}
\begin{table}[H] 
\renewcommand{\arraystretch}{1.5}
\begin{tabular}{ |c|c|c| } 
 \hline
Triangle group & Value of $(a,b,c)$ for which $E(a,b,c)=\{0\}$ \\[7pt] 
\hline 
 $(\infty,\infty,\infty)$ & $(\frac{1}{2}, \frac{1}{2},1)$ \\[7pt] 
 \hline
  $(2,\infty,\infty)$ & $(\frac{1}{4}, \frac{1}{4},1), (\frac{3}{4}, \frac{3}{4},1), (\frac{1}{4}, \frac{3}{4},1), (\frac{3}{4}, \frac{3}{4},\frac{3}{2})$ \\[7pt]  
 \hline
   $(3,\infty,\infty)$ & $(\frac{1}{3}, \frac{1}{3},1), (\frac{2}{3}, \frac{2}{3},1), (\frac{1}{3}, \frac{2}{3},1), (\frac{2}{3}, \frac{2}{3},\frac{4}{3})$ \\[7pt]  
 \hline
    $(3,3,\infty)$ & $(\frac{1}{6}, \frac{1}{2},1), (\frac{1}{2}, \frac{5}{6},1), (\frac{1}{2}, \frac{1}{2},\frac{4}{3}), (\frac{1}{2}, \frac{5}{6},\frac{4}{3}), (\frac{5}{6}, \frac{5}{6},\frac{4}{3})$ \\[7pt]  
 \hline
   $(4,4,\infty)$ & $(\frac{1}{4}, \frac{1}{2},1), (\frac{1}{2}, \frac{3}{4},1), (\frac{1}{2}, \frac{1}{2},\frac{5}{4}), (\frac{1}{2}, \frac{3}{4},\frac{5}{4}), (\frac{3}{4}, \frac{3}{4},\frac{5}{4})$ \\[7pt]  
 \hline
   $(6,6,\infty)$ & $(\frac{1}{3}, \frac{1}{2},1), (\frac{1}{2}, \frac{2}{3},1), (\frac{1}{2}, \frac{1}{2},\frac{7}{6}), (\frac{1}{2}, \frac{2}{3},\frac{7}{6}), (\frac{2}{3}, \frac{2}{3},\frac{7}{6})$ \\[7pt]  
 \hline
  $(2,4,\infty)$ & $(\frac{1}{8}, \frac{3}{8},1), (\frac{5}{8}, \frac{7}{8},1), (\frac{1}{8}, \frac{5}{8},1), (\frac{3}{8}, \frac{7}{8},1),  (\frac{3}{8},\frac{3}{8}, \frac{5}{4}),$ \\[7pt]  
    & $(\frac{7}{8},\frac{7}{8}, \frac{5}{4}), (\frac{3}{8},\frac{7}{8}, \frac{5}{4}), (\frac{5}{8},\frac{5}{8}, \frac{3}{2}), (\frac{7}{8},\frac{7}{8}, \frac{3}{2}), (\frac{7}{8},\frac{5}{8}, \frac{3}{2})$ \\[7pt]  
 \hline
  $(2,6,\infty)$ & $(\frac{1}{6}, \frac{1}{3},1), (\frac{5}{6}, \frac{2}{3},1), (\frac{1}{6}, \frac{2}{3},1), (\frac{1}{3}, \frac{5}{6},1),  (\frac{1}{3},\frac{1}{3}, \frac{7}{6}),$ \\[7pt]  
    & $(\frac{5}{6},\frac{5}{6}, \frac{7}{6}), (\frac{5}{6},\frac{1}{3}, \frac{7}{6}), (\frac{2}{3},\frac{2}{3}, \frac{3}{2}), (\frac{5}{6},\frac{5}{6}, \frac{3}{2}), (\frac{5}{6},\frac{2}{3}, \frac{3}{2})$ \\[7pt]  
 \hline

 \end{tabular}
 \vspace{0.5cm}
\caption{Data with trivial Exceptional Set}
\label{table:0}
\end{table}
\end{center}

\section{Exceptional Sets for Takeuchi's Class II} \label{sec:ex2}
$$
  \begin{diagram}
  \node[2]{(2,4,6)}  \arrow{sw,l,-}{2}  \arrow{se,l,-}{2} \\
  \node[1]{(2,6,6)}    \arrow{s,l,-}{2} 
   \node[2]{(3,4,4)}   \\
   \node[1]  {(3,6,6)} 
  \end{diagram}
 $$
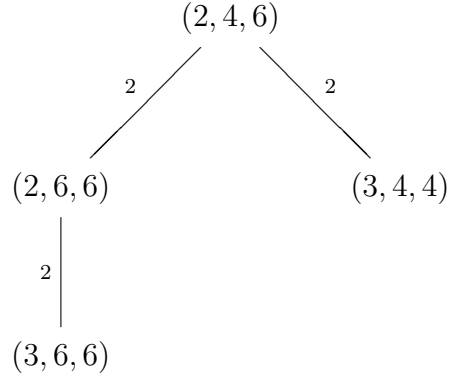
  \begin{center}Figure 2: Takeuchi's Class II \end{center}
As in class I, the hypergeometric function can be viewed as a quotient of periods on suitable subvarieties of Jacobians through Euler's integral representation, as described in Section~\ref{sec:basic}. Consequently, the transcendence results of Wolfart and W\"ustholz can again be used to obtain necessary conditions on algebraic values of $_2F_1(a,b,c;z)$.
\par The difficulty is that there is less known about the explicit construction of modular forms. In class $1$, the corresponding modular objects are automorphic forms on the modular curves. In contrast, for class $2$, these objects are automorphic forms on the Shimura curve $X_0^D(N)$ associated to an Eichler order of level $N$ in a quaternion algebra of discriminant $D$ over $\Q$.
\par In~\cite{yang2018special}, Yang has shown that 
\begin{equation}\label{yang}
_2F_1\begin{bmatrix}
\begin{matrix} \frac{1}{24} \\ \\ \end{matrix} & \begin{matrix} \frac{5}{24} \\ \frac{3}{4}\end{matrix} & ; \: s(\tau_d) 
\end{bmatrix} \in \frac{\Omega_{-d}}{\Omega_{-4}} \cdot \overline{\Q},  
\;\;\;\;\;
_2F_1\begin{bmatrix}
\begin{matrix} \frac{1}{24} \\ \\ \end{matrix} & \begin{matrix} \frac{7}{24} \\ \frac{5}{6}\end{matrix} & ; \: \frac{1}{s(\tau_d)} 
\end{bmatrix} \in \frac{\Omega_{-d}}{\Omega_{-3}} \cdot \overline{\Q}.  
\end{equation}
where $s$ is the Hauptmoduln of $X_0^6(1)/W_6$ taking values $0,1,\infty$ at the CM points of discriminant $-4,-24$ and $-3$ respectively, $\tau_d$ is a CM point of discriminant $-d$, and $\Omega_{-d}$ is a Chowla-Selberg period defined as
$$\Omega_{-d} = \sqrt{\pi} \Big(\prod_{i=1}^{d-1} \Gamma(\frac{i}{d})^{\chi_{-d}(i)} \Big)^{\frac{1}{2h^{\prime}(-d)}}$$
where $\chi_{-d}(n)$ is the Jacobi symbol $\legendre(-d,n)$, $h^{\prime}(-d)$ is the class number $h(D)$ of $\Q(\sqrt{-d})$ divided by the number of roots of unity in this field. The above facts~\eqref{yang} are consistent with Theorem~\ref{tmain} of this paper. From the identities in Table~\ref{table:identity} and~\cite[Corollary 5.10.4]{cohen2017modular} we can also explicitly write down Theorem~\ref{tmain} as follows. 
\begin{theorem}
Let $(a,b,c)$ be a datum satisfying $0<a,b,c<1$ such that the monodromy group $\Delta$ associated to $\text{HDE}(a,b,c;z)$ is an arithmetic triangle group in Takeuchi's class I (see Figure 1). Let $t$ be the Hauptmodul as in Table~\ref{table:2} for the corresponding group $\Delta$, and $d \in \{1,2,3\}$ be the corresponding positive integer as in Theorem~\ref{tmain}. Then for a CM point $\tau_{\mathfrak{d}}$ of discriminant $-\mathfrak{d}$, 
$$
_2F_1\begin{bmatrix}
\begin{matrix} a \\ \\ \end{matrix} & \begin{matrix} b \\ c \end{matrix} & ; \: t(\tau_d) 
\end{bmatrix} \in \frac{\Omega_{-\mathfrak{d}}}{\Omega_{-d}} \cdot \overline{\Q}. 
$$
\end{theorem}
\par While proving~\eqref{yang}, Yang realized the meromorphic modular forms on Shimura curves as Borechards forms. Using Schofer's formula, $_2F_1$ at the CM points of Borechards forms can be expressed as a multiple of Chowla-Selberg periods. In particular, Yang showed that if $d=-4r^2$ for some integer $r$, the hypergeometric function $_2F_1(\frac{1}{24}, \frac{5}{24}, \frac{3}{4};s(\tau_d)) \in \overline{\Q}$. 
\par Now we define the special exceptional set for $\kappa \in \C$ as
$$E^\kappa(a,b,c) \coloneqq \{z \in \overline{\Q} \mid\; _2F_1(a,b,c;z) \in \kappa\overline{\Q} \}.$$ Baba and Granath showed that CM points of discriminant $-d$ of certain fractional linear transformations of the modular $j$-function belong to $E^{(\Omega_{-d}/\Omega_{-x})}(a,b,c)$ where $(a,b,c)$ is the datum coming from the triangle group $(2,4,6)$ and $x = \text{L.C.D.}(a,b,c)$. See~\cite[Theorem 1]{baba2015quaternionic} for details. The authors studied the modular forms for the triangle group $(2,4,6)$ related to abelian surfaces with quaternionic multiplication.
\par In class 1, Schneider's theorem gives a necessary and sufficient condition for the Hauptmoduln to be algebraic. In general, there is no such result for class II. Although in~\cite{shiga2007algebraic} Shiga and Wolfart give sufficient conditions for which the Schwarz triangle function is algebraic.
\par The exceptional set for class II datum  $(\frac{1}{24},\frac{5}{24},\frac{3}{4})$, is determined since $B(\frac{5}{24}, \frac{13}{24})$ is an algebraic multiple of $B(\frac{1}{2}, \frac{1}{4})$, making the dimension of corresponding variety $1$. This is the $\Omega_{-4}$ family. Therefore, the result aligns with the findings presented in~\cite{yang2018special}, which we utilize for the converse, along with the findings from class 1, i.e., when the Chowla-Selberg ratio is algebraic. The same applies to the datum $(\frac{1}{24},\frac{7}{24},\frac{5}{6})$, which is $\Omega_{-3}$ family. 
So in class II, the structure of exceptional sets remains the same. For a triplet $(a,b,c)$, the exceptional set, if it is non-zero, is of the form
$$E(a,b,c) = \{z \in \overline{\Q} \mid \exists \;\tau_d \in \Q(\sqrt{-d}) \cap \UH \;\;\text{with}\;\;z = t(\tau_d)\}$$
where $t$ is a Hauptmodul for the corresponding triangle group.

\begin{center}
\begin{table}[H]
\renewcommand{\arraystretch}{1.5}
\begin{tabular}{|w{c}{3cm}|w{c}{3cm}|w{c}{3cm}|w{c}{3cm}|}
\hline
Triangle group & $(a,b,c)$ & $\Q(\sqrt{-d})$ & $t$ \\[7pt]
\hline
\multirow{2}{*}{$(2, 4, 6)$} &  $(\frac{1}{24}, \frac{5}{24}, \frac{3}{4})$ & $\Q(i)$ & $s$ \\[7pt]
\cline{2-4} 
&  $(\frac{1}{24}, \frac{7}{24}, \frac{5}{6})$ & $\Q(i\sqrt{3})$ & $\frac{1}{s}$ \\[7pt]
\hline
\end{tabular}
\end{table}
\end{center}

\begin{remark}
The variety-theoretic argument is not restricted to Takeuchi's class I. For a general datum $(a,b,c)$ whose associated monodromy group is an arithmetic triangle group, we can study the abelian variety $T_0=\operatorname{Jac}_{\mathrm{new}}(X(M,0))$ from the eigenspace decomposition of its holomorphic differentials. In an ongoing work, we carry out this computation for data arising from all of Takeuchi's classes and determine whether the corresponding abelian varieties are simple and, when they are not, describe their decomposition up to isogeny. This gives the variety-theoretic description required for applying the transcendence results of Wolfart and W\"ustholz, and hence for studying the necessary condition for $z\in E(a,b,c)$.
\end{remark}

\begin{remark}
The converse direction is more difficult for the remaining arithmetic triangle groups. In class I, the explicit hypergeometric--modular identities of Section~3, together with the description of the corresponding Hauptmoduln, allow us to obtain the sufficient condition from CM points. For the arithmetic triangle groups in the other Takeuchi classes, such explicit modular descriptions are in general not known in a form suitable for this argument. Thus, although the associated abelian varieties can be studied for all classes, determining the explicit description of exceptional sets requires further information on the modularity of the corresponding hypergeometric functions and their Hauptmoduln.
\end{remark}

\medskip
\textbf{Role of AI in this work.} The work for this project was started in 2025, and the mathematical arguments and proofs presented here were completely developed and written by the author. No AI model has been used for any mathematical help in this work. However, publicly available ChatGPT-5.3 was used to improve the exposition and language of the manuscript. 

\newpage
\appendix
\section{Table determining exceptional sets}

\begin{center}
\begin{table}[H]
\renewcommand{\arraystretch}{1.15}
\begin{tabular}{|w{c}{3cm}|w{c}{3cm}|w{c}{3cm}|w{c}{3cm}|}
\hline
Triangle group & $(a,b,c)$ & $\Q(\sqrt{-d})$ & $ t$ \\[7pt]
\hline
$(2, \infty, \infty)$ & $(\frac{1}{4}, \frac{1}{4}, \frac{1}{2})$ & $\Q(i)$ & $S_2$ \\[7pt]
\hline
$(3, \infty, \infty)$ & $(\frac{1}{3}, \frac{1}{3}, \frac{2}{3})$ & $\Q(i\sqrt{3})$ & $S_3$ \\[7pt]
\hline
\multirow{3}{*}{$(3, 3, \infty)$} &  $(\frac{1}{6}, \frac{1}{6}, \frac{2}{3})$ & $\Q(i\sqrt{3})$ & $R_3$ \\[7pt]
\cline{2-4} 
&  $(\frac{1}{2}, \frac{1}{2}, \frac{2}{3})$ & $\Q(i\sqrt{3})$ & $R_3$ \\[7pt]
\cline{2-4} 
&  $(\frac{1}{6}, \frac{1}{2}, \frac{2}{3})$ & $\Q(i\sqrt{3})$ & $\frac{R_3}{R_3-1}$ \\[7pt]
\hline
\multirow{3}{*}{$(4, 4, \infty)$} &  $(\frac{1}{4}, \frac{1}{4}, \frac{3}{4})$ & $\Q(i)$ & $R_4$ \\[7pt]
\cline{2-4} 
&  $(\frac{1}{2}, \frac{1}{2}, \frac{3}{4})$ & $\Q(i)$ & $R_4$ \\[7pt]
\cline{2-4} 
&  $(\frac{1}{4}, \frac{1}{2}, \frac{3}{4})$ & $\Q(i)$ & $\frac{R_4}{R_4-1}$ \\[7pt]
\hline
\multirow{3}{*}{$(6, 6, \infty)$} &  $(\frac{1}{3}, \frac{1}{3}, \frac{5}{6})$ & $\Q(i\sqrt{3})$ & $R_6$ \\[7pt]
\cline{2-4} 
&  $(\frac{1}{2}, \frac{1}{2}, \frac{5}{6})$ & $\Q(i\sqrt{3})$ & $R_6$ \\[7pt]
\cline{2-4} 
&  $(\frac{1}{3}, \frac{1}{2}, \frac{5}{6})$ & $\Q(i\sqrt{3})$ & $\frac{R_6}{R_6-1}$ \\[7pt]
\hline
\multirow{6}{*}{$(2, 4, \infty)$} &  $(\frac{1}{8}, \frac{1}{8}, \frac{3}{4})$ & $\Q(i)$ & $J_4$ \\[7pt]
\cline{2-4} 
&  $(\frac{5}{8}, \frac{5}{8}, \frac{3}{4})$ & $\Q(i)$ & $J_4$ \\[7pt]
\cline{2-4} 
&  $(\frac{1}{8}, \frac{5}{8}, \frac{3}{4})$ & $\Q(i)$ & $\frac{J_4}{J_4-1}$ \\[7pt]
\cline{2-4} 
&  $(\frac{1}{8}, \frac{1}{8}, \frac{1}{2})$ & $\Q(i\sqrt{2})$ & $1-J_4$  \\[7pt]
\cline{2-4} 
&  $(\frac{3}{8}, \frac{3}{8}, \frac{1}{2})$ & $\Q(i\sqrt{2})$ & $1-J_4$  \\[7pt]
\cline{2-4} 
&  $(\frac{1}{8}, \frac{3}{8}, \frac{5}{6})$ & $\Q(i\sqrt{2})$ & $\frac{J_4-1}{J_4}$\\[7pt]
\hline
\multirow{6}{*}{$(2, 6, \infty)$} &  $(\frac{1}{6}, \frac{1}{6}, \frac{5}{6})$ & $\Q(i\sqrt{3})$ & $J_6$ \\[7pt]
\cline{2-4} 
&  $(\frac{2}{3}, \frac{2}{3}, \frac{5}{6})$ & $\Q(i\sqrt{3})$ & $J_6$ \\[7pt]
\cline{2-4} 
&  $(\frac{1}{6}, \frac{2}{3}, \frac{5}{6})$ & $\Q(i\sqrt{3})$ & $\frac{J_6}{J_6-1}$ \\[7pt]
\cline{2-4} 
&  $(\frac{1}{6}, \frac{1}{6}, \frac{1}{2})$ & $\Q(i\sqrt{3})$ & $1-J_6$  \\[7pt]
\cline{2-4} 
&  $(\frac{1}{3}, \frac{1}{3}, \frac{1}{2})$ & $\Q(i\sqrt{3})$ & $1-J_6$  \\[7pt]
\cline{2-4} 
&  $(\frac{1}{3}, \frac{1}{6}, \frac{1}{2})$ & $\Q(i\sqrt{3})$ & $\frac{J_6-1}{J_6}$\\[7pt]
\hline
\end{tabular}
\vspace{0.5cm}
\caption{Description of Exceptional Sets for Class I}
\label{table:main}
\end{table}
\end{center}

\begin{landscape}
\centering\appendix{Table with explicit examples of algebraic values}

\thispagestyle{empty}
\begin{table}[H]
\centering
\small
\renewcommand{\arraystretch}{2.25}
\begin{tabular}{|c|c|c|c|c|}
\hline
Triangle group & Datum $(a,b,c)$ & $\tau$ & $t(\tau)$ & ${}_2F_1(a,b,c;t(\tau))$ \\
\hline
$(2,\infty,\infty)$
& $\left(\frac14,\frac14,\frac12\right)$
& $i$
& $9$
& $\dfrac{2-i}{2\sqrt2}$ \\
\hline
$(3,\infty,\infty)$
& $\left(\frac13,\frac13,\frac23\right)$
& $\sqrt{3}i$
& $2\left(329+261\sqrt[3]{2}+207\sqrt[3]{4}\right)$
& $\dfrac{5-3i\sqrt3}{6}\left(2^{-1/3}-2^{-2/3}\right)$ \\
\hline
$(3,3,\infty)$
& $\left(\frac16,\frac16,\frac23\right)$
& $\sqrt{3}i$
& $\dfrac12-\dfrac{11i}{4}$
& $\dfrac{\sqrt[3]{2}}{2\sqrt6}(\sqrt3-2i)(1+i)$  \\
\hline
$(4,4,\infty)$
& $\left(\frac14,\frac14,\frac34\right)$
& $i$
& $\dfrac{8-7i\sqrt2}{16}$
& $\dfrac{2^{7/8}\sqrt5}{4\sqrt{100-70\sqrt2}}
\left(1-i(5\sqrt2-7)\right)$ \\
\hline
$(6,6,\infty)$
& $\left(\frac13,\frac13,\frac56\right)$
& $\sqrt3 i$
& $\dfrac12-\dfrac{i}{6}\sqrt{1472+1174\sqrt[3]{2}+933\sqrt[3]{4}}$
& $\dfrac{(\sqrt[3]{2}-1)^{1/3}}{6^{2/3}}
\left(2\sqrt 3-i\right)$\\
\hline
\multirow{2}{*}{$(2,4,\infty)$}
& $\left(\frac18,\frac18,\frac34\right)$
& $2i$
& $\dfrac{17901}{32}+\dfrac{50787}{128}\sqrt2$
& 
$\dfrac{(3\sqrt2-4)^{1/8}}{2^{7/2}}
\left[
(5+i)\sqrt{2+\sqrt2}+(1-5i)\sqrt{2-\sqrt2}
\right]$ \\
\cline{2-5}
& $\left(\frac18,\frac18,\frac12\right)$
& $\sqrt2 i$
& $-\dfrac{425+325\sqrt2}{32}$
& $\dfrac{3\,\cdot\,2^{3/8}}{4(1+\sqrt2)^{1/8}}$ \\
\hline
\multirow{2}{*}{$(2,6,\infty)$}
& $\left(\frac16,\frac16,\frac56\right)$
& $\sqrt3 i$
& $\dfrac{1481}{9}+\dfrac{1174}{9}\sqrt[3]{2}+\dfrac{311}{3}\sqrt[3]{4}$
& $\dfrac{(\sqrt[3]{2}-1)^{1/3}}{6^{2/3}}
\left(2\sqrt 3-i\right)$ \\
\cline{2-5}
& $\left(\frac16,\frac16,\frac12\right)$
& $\sqrt3 i$
& $-\dfrac{1472}{9}-\dfrac{1174}{9}\sqrt[3]{2}-\dfrac{311}{3}\sqrt[3]{4}$
& $\dfrac{2(\sqrt[3]{2}-1)^{1/6}}{\sqrt3\sqrt{1+\sqrt[3]{2}}}$ \\
\hline
\end{tabular}
\caption{Explicit examples of algebraic hypergeometric values at CM points of the corresponding Hauptmoduln.}
\label{table:explicit_examples}
\end{table}

The example from the first row is illustrated in~\ref{appendixexample}. The same procedure is used to find these algebraic values. 
\end{landscape}
\newpage

\bibliographystyle{plain}
\bibliography{mathscinet_ref}

\end{document}